\documentclass[12pt, reqno]{amsart}

\usepackage{xcolor}
\usepackage[colorlinks=true, citecolor=red]{hyperref}
\usepackage{fullpage}
\usepackage{amssymb} 
\usepackage[normalem]{ulem}

\makeatletter
\@namedef{subjclassname@2020}{%
  \textup{2020} Mathematics Subject Classification}
\makeatother

\newcommand{\GL}{\mathrm{GL}}
\newcommand{\SL}{\mathrm{SL}}
\newcommand{\Z}{\mathbb{Z}}
\newcommand{\R}{\mathbb{R}}
\renewcommand{\d}{\mathrm{d}}
\newcommand{\dd}{\mathrm{d}}
\newcommand{\sums}{\ \sideset{}{^*}\sum}
\renewcommand{\leq}{\leqslant}
\renewcommand{\geq}{\geqslant}
\renewcommand{\mod}{\operatorname{mod}}
\renewcommand{\pmod}{\mod}

\renewcommand{\Re}{\operatorname{Re}}
\newcommand{\Nmain}{\mathcal{N}_{\text{main}}}
\newcommand{\Ncomp}{\mathcal{N}_{\text{comp}}}
\newcommand{\Smain}{S_{\text{main}}(N)}
\newcommand{\Smainprime}{S_{\text{main}}'}
\newcommand{\Scomp}{S_{\text{comp}}(N)}

\newtheorem{lemma}{Lemma}
\newtheorem{theorem}{Theorem}

\title{Subconvexity for twisted $\GL_3$ L-functions}
\author{Eren Mehmet Kıral\textsuperscript{1}, Chan Ieong Kuan\textsuperscript{2} and Didier Lesesvre}
\date{\today}

\thanks{1. The first-named author is a Special Postdoctoral Researcher (SPDR) at RIKEN. Part of this work was supported by Grant-in-Aid for JSPS Research Fellow 18F18326.}
\thanks{2. The second-named author is supported in part by NSFC (No.11901585).}



\address{\newline RIKEN Advanced Intelligence Project (AIP), Mathematical Sciences Team\newline
Special Postdoctoral Researchers Program
 \newline
Nihonbashi 1-chome Mitsui Building, 15th floor,\newline
1-4-1 Nihonbashi, Chuo-ku, Tokyo, 103-0027, Japan}
\email{erenmehmetkiral@protonmail.com}

\address{\newline School of mathematics (Zhuhai) \newline
Zhuhai Campus, Sun Yat-Sen University \newline
Tangjiawan, Zhuhai, Guangdong, 519082, China (PRC)}
\email{didier@mail.sysu.edu.cn, lesesvre@math.cnrs.fr}
\email{kuanchi3@mail.sysu.edu.cn}

\numberwithin{equation}{section}

\begin{document}

\begin{abstract}
Using the circle method, we obtain subconvex bounds for $\GL_3$ $L$-functions twisted by a character~$\chi$ modulo a prime $p$, hybrid in the $t$ and $p$-aspects.
\end{abstract}

\maketitle 

\section{Introduction}
\label{sec:introduction}

\subsection{Statement of Results}
\label{sec:introduction-results}

Let $\pi$ be a Hecke-Maass form for $\SL_3(\Z)$ and $\chi$ a primitive character modulo a {prime} $p$. If $p=1$, then $\chi$ is the trivial character. The twisted $L$-function is defined by
\begin{equation}
	L(s, \pi \times \chi) = \sum_{n=1}^\infty \frac{\lambda(1,n) \chi(n)}{n^s}
\end{equation}
for $\Re(s)>1$. This function has an analytic continuation to the entire plane, and satisfies a functional equation. The analytic conductor is asymptotically of size $c(t,\pi \times \chi)~\asymp~(pt)^3$, \emph{for a fixed~$\pi$}, see for example \cite{iwaniec_perspectives_2000}. We are interested in bounds of the form
\begin{equation}\label{eq:subconvexityalpha}
L\left(\tfrac{1}{2}+it, \pi\times\chi\right) \ll_\varepsilon (pt)^{\frac{3}{4}-\delta + \varepsilon}
\end{equation}
for a certain $0 \leq \delta \leq \tfrac{3}{4}$ and for any $\varepsilon>0$. The bound with $\delta = 0$ follows from the functional equation and the Phragm\'{e}n-Lindel\"{o}f convexity principle for all automorphic $L$-functions \cite{iwaniec_analytic_2004} and is called the convexity bound. The Lindelöf hypothesis is the statement with $\delta = \frac{3}{4}$. Any improvement $\delta > 0$ is called a subconvex bound; particularly challenging milestones are the Burgess-type bound $\delta = \tfrac{3}{16}$ and the Weyl-type bound $\delta = \frac{1}{4}$ giving bounds \eqref{eq:subconvexityalpha} which are respectively three-fourths and two-thirds of the convexity exponent. In this paper we show the following theorem.

\begin{theorem}
\label{theorem}
Let $\pi$ be a Hecke-Maass form for $\SL_3(\Z)$ and $\chi$ a primitive character modulo a prime $p$. Assume $p<t^{{8/7}}$.
The following subconvex bound holds:
\begin{equation}
L\left(\tfrac{1}{2}+it, \pi\times\chi\right) \ll_\varepsilon (pt)^{\frac{3}{4} - {\frac{3}{40}} + \varepsilon}.
\end{equation}
\end{theorem}

\noindent For $p=1$ this {exponent matches the recent result of \cite{aggarwal2019new} which improved} upon the then best known subconvexity bound in the $t$-aspect, due to Munshi~\cite{munshi_circle_2014} with $\delta = \frac{1}{16}$. It also improves,  in the allowed range for $(t,p)$, the best known subconvexity bound in the hybrid $(t,p)$-aspect, due to Lin \cite{lin_bounds_2020} with $\delta = \frac{1}{36}$.  

\subsection{Previous Results}

The literature on the subconvexity problem is rich. We start with mentioning some results in the $\GL_1$ and $\GL_2$ cases for the purpose of assessing the strength of current results for $\GL_3$ $L$-functions. The Weyl-type bound has  been reached in full generality for $\GL_1$ $L$-functions \cite{petrow_fourth_2019} in the hybrid $(t, \chi)$-aspect. Subconvexity has been established in all aspects simultaneously for $\GL_2$ $L$-functions by Michel and Venkatesh \cite{michel_subconvexity_2010} with unspecified exponent, recently determined as $\delta = \tfrac{1}{128}$ by Wu \cite{wu_explicit_2021}. For $(t,\chi)$-aspect, the Burgess-type subconvex bound has been reached on both aspects simultaneously by Wu \cite{Wu_Burgess_2014} conditionally on the Ramanujan  hypothesis. For holomorphic cusp forms, the Burgess-type subconvex bound in $\chi$-aspect and Weyl-type subconvex bound in $t$-aspect has been achieved simultaneously by the second author \cite{Kuan_hybrid_2018}, also conditionally.  

In the case of $\GL_3$ self-dual $L$-functions, Li \cite{li_bounds_2011} achieved subconvexity with $\delta = \frac{1}{16}$ in the $t$-aspect via proving a first moment result for a family of $\GL_3 \times \GL_2$ $L$-functions. The proof relies on positivity of the central values of the $L$-functions involved, allowing one to deduce bounds for a single $L$-function by dropping all but one term. Since then, there are quite a few results based on this approach and its amplified variants. We list a few here: Blomer~\cite{blomer_subconvexity_2012}  achieved subconvexity in the $\chi$-aspect with $\delta =  \frac{1}{8}$ for quadratic characters and Nunes~\cite{nunes_subconvexity_2017} did so in the $t$-aspect with $\delta =  \frac{1}{8}$. For $(t,\chi)$-aspect, a similar approach led Huang \cite{huang_hybrid_2021} to reach subconvexity with $\delta = \tfrac{1}{46}$. 

Self-duality is not a common feature among $\GL_3$ $L$-functions, so it is desirable to remove this condition. Munshi has successfully done so by deploying circle method, where he starts from a single $L$-function instead of a moment. He achieves subconvexity in the $t$-aspect in~\cite{munshi_circle_2014} with $\delta = \frac{1}{16}$ and this exponent has been recently improved by Aggarwal \cite{aggarwal2019new} to $\delta=\tfrac{3}{40}$. Using the $\GL_2$ Petersson trace formulas as an expansion for the delta-symbol, he also succeeded in establishing $\delta = \frac{1}{308}$ for the $\chi$-aspect in \cite{munshi_twists_2016}. Inspired by this work, Holowinsky and Nelson wrote the character as a weighted sum of exponentials and Kloosterman sums obtaining $\delta = \frac{1}{36}$ in the $\chi$-aspect in \cite{holowinsky_subconvex_2018}. With this methodology, afterwards Lin gave a hybrid subconvex bound in the $(t,\chi)$-aspect with $\delta = \frac{1}{36}$  in \cite{lin_bounds_2020}, however failing to reach the exponent $\delta = \tfrac{3}{40}$. 

In this paper, we take upon a suggestion made by Munshi in \cite{munshi_subconvexity_2018} that simplifies the treatment of certain oscillatory integrals in \cite{munshi_circle_2014}. We also consider a character twist which ended up interacting delicately with the $t$-aspect circle method, \textit{ibid}. {We improve the best known bound in the joint $(t,\chi)$-aspect and achieve $\delta = \tfrac{3}{40}$, in a restricted range for~$(t, p)$.

\smallskip

\noindent \textit{Remark.} During preparation of this manuscript we came across Huang and Xu's very recent work~\cite{huang2021hybridnew}. They prove our Theorem \ref{theorem} with the same $t$ aspect savings of $\delta_t = -3/40$ but with a worse~$p$ aspect savings of $\delta_p = -1/32$. Their proof does not use the conductor lowering trick, but a method called \emph{mass transform}. 

\subsection*{Structure of the paper}

In Section \ref{sec:approximate-functional-equation} we are left with estimating a finite sum after using the approximate functional equation. In Section \ref{sec:conductor-lowering},  we apply the conductor lowering trick of Munshi as in \cite{munshi_subconvexity_2018} in order to reduce  the range of the variables introduced by the delta method. We next apply the delta method in Section \ref{sec:delta} and decouple the oscillations of the~$\GL_3$ automorphic coefficients $\lambda(1,n)$ from those of $\chi(n)n^{-it}$. 

Once these oscillations are separated, we apply the Poisson summation formula on the $\GL_1$-sum in Section \ref{sec:GL1-poisson} and the $\GL_3$-Vorono\"{i} formula on the $\GL_3$-sum in Section~\ref{sec:GL3-poisson}. The integral transforms appearing in these two summation formulas feature a common variable, coming from the particular form of the delta method used. This is the point that allows for a simplified treatment of the integrals appearing in \cite{munshi_circle_2014}; it is carried out in Section~\ref{subsec:x-integral}. 

After an application of Cauchy-Schwarz inequality and Poisson formula in Section \ref{sec:cauchy}, the estimates boil down to estimating an arithmetic and an analytic parts. The analytic part consists in various oscillatory integrals and is taken care of by the stationary phase method in Section \ref{sec:analytic} while the arithmetic part is taken care of in Section \ref{sec:arithmetic}. These bounds and an optimization in the conductor lowering parameter then allows to reach the subconvexity result in Section~\ref{subsec:subconvexity}.

\subsection*{Notations}

We use the following usual notations: $n \sim N$ if $n \in (N,2N)$, $f \ll g$ if there is a constant $C>0$ such that $|f| \leqslant C |g|$, and $f \asymp g$ if $f \ll g$ and $g \ll f$. All these asymptotic relations are relative to $t \to \infty$ and $p \to \infty$. Also, as is common in the literature, our use of $\varepsilon$ is fluid. It refers to an arbitrarily small positive exponent but may change from line to line. We let $f \ll_\varepsilon g$ if $f \ll (pt)^\varepsilon g$, and accordingly for $\asymp_\varepsilon$. Since the subconvexity bounds are of the form \eqref{eq:subconvexityalpha} and always feature an arbitrarily small exponent $\varepsilon$, this convention is natural and lightens notations considerably. We afford to drop the subscript in indices of summations.

\subsection*{Acknowledgments} We are grateful to Ritabrata Munshi for enlightening discussions.

\section{Setting}

\subsection{Approximate functional equation}
\label{sec:approximate-functional-equation}

The approximate functional equation \cite[Theorem 5.3]{iwaniec_analytic_2004} followed by standard calculations leads to
\begin{equation}
	L(\tfrac{1}{2} + it, \pi \times \chi) \ll_\varepsilon  \sup_{N \ll (pt)^{3/2}} \frac{S(N)}{N^{1/2}} + O_A((pt)^{-A}),
\end{equation}
for any $A>0$. Here
\begin{equation}\label{eq:SNdefinition}
	S(N) = \sum_{n =1}^\infty \lambda(1, n) n^{-it} \chi(n)  W\left(\frac{n}{N}\right)
\end{equation}
with $W$ a fixed smooth nonnegative bump function supported in $[1,2]$ and $\lambda(1,n)$ are the Fourier-Hecke coefficients of the $\GL_3$ form $\pi$. Assuming Lindel\"of on average, applying Cauchy-Schwarz and then bounding trivially, we have $S(N) \ll_\varepsilon N \ll_\varepsilon N^{1/2} (pt)^{3/4}$, which corresponds to the convexity bound. Therefore our goal amounts to obtaining any extra savings on this trivial bound at this point. 

\subsection{Conductor lowering}
\label{sec:conductor-lowering}

Let us formally separate the $\GL_1$ and the $\GL_3$ oscillations: 
\begin{equation}
	S(N) = \sum_{n,m = 1}^\infty \delta_{n=m} \lambda (1,n) m^{-it} \chi(m) U\left(\frac{n}{N} \right) U\left(\frac{m}{N}\right) .
\end{equation}
Here $U = W^{\frac{1}{2}}$, which is again a smooth bump function supported in $[1,2]$. 
The circle method is an analytic expansion of the $n=m$ condition, which is valid only if $n-m$ is in a restricted range \cite{munshi_subconvexity_2018}. It will turn out to be essential in our final bound if we can even further restrict this gap, after \emph{opening up} $\delta_{n=m}$ via the circle method.

For this goal we now introduce Munshi's \textit{conductor lowering} procedure. Let $V$ be a function supported in $[1,2]$ such that $\int_{\R}V(v) \d v = 1$. Let $K>1$ be a parameter to be determined later. Then we write 
\begin{equation}
\label{eq:S(N)-original}
	S(N) =  \sum_{n, m = 1}^\infty  \delta_{n = m} \lambda(1,n)  m^{-it} \chi(m) U\left(\frac{n}{N} \right) U\left(\frac{m}{N}\right) \frac{1}{K} \int_{\R} V\left(\frac{v}{K} \right) \left(\frac{n}{m}\right)^{iv}  \d v .
\end{equation}
The innermost integral in $v$, taken separately from the $n = m$ condition, ensures that $|n - m| \ll_\varepsilon N/K$. This is the content of the following lemma. 

\begin{lemma}
\label{lem:conductor-lowering}
	If $n,m\sim N$ and $V$ is a smooth function supported in $[1,2]$, then we have that
\begin{equation}
		\frac{1}{K} \int_{\R} V\left(\frac{v}{K} \right) \left(\frac{n}{m}\right)^{iv}  \d v \ll_A \left(\frac{K |n - m| }{m}\right)^{-A} 
\end{equation}

\noindent for any $A > 0$. Thus unless $|n- m| \ll_\varepsilon  \frac{N}{K}$ we have that the given integral is $O_A(t^{-A})$ for any $A>0$.  
\end{lemma}

\begin{proof}
	Let us assume $n > m$ by exchanging $n$ and $m$ if necessary. Then by a change of variables the integral is given by
\begin{equation}
		 \int_{\R} V(v) e^{ \pm i v K \log \left( 1 + \tfrac{n-m}{m}\right)} \d v.
\end{equation}
Let us call $h: = n-m$. After integration by parts $k$ times, we see that this integral is asymptotically bounded by 
\begin{equation}
		(K \log (1 + \tfrac{h}{m}))^{-k}. 
\end{equation}\smallskip
Note that $\tfrac{h}{m} \in [0,1]$, and in this region $\log(1 + \tfrac{h}{m}) \asymp \tfrac{h}{m}$. Thus if $K \log ( 1 + \tfrac{h}{m}) \approx \tfrac{Kh}{m}  \gg (p t)^{\varepsilon}$, then the integral is bounded by $O_k(t^{-k\epsilon})$ for all integer $k$. Otherwise  $\tfrac{Kh}{m} \ll (pt)^{\varepsilon}$, i.e.
\begin{equation}
		h \ll_\varepsilon \frac{m}{K}  \asymp \frac{N}{K} .
\end{equation}\vspace{-0.2cm}
This finishes the proof of the lemma.
\end{proof}

Let us now assume 
\begin{equation}
K \ll (pt)^{-\varepsilon} \min\left.\{\sqrt{pt}, t \right\}.
\end{equation} This allows us to ignore certain terms, but is also the reason why Theorem~\ref{theorem} displays a restricted range for $(t,p)$.

\subsection{Delta method}
\label{sec:delta}

We analytically separate the $\lambda(1,n)$ from the $\chi(n)n^{-it}$ using the delta method \cite{munshi_subconvexity_2018} which we use in the form\smallskip
\begin{equation} \label{eq:delta}
	\delta_{h = 0} = \frac{1}{Q} \sum_{q \leqslant Q} \,\, \frac{1}{q} \: \sums_{a \text{ mod }q } e\left(\frac{ha}{q}\right) \int_{\R} g(q,x) e \left(\frac{hx}{qQ}\right) \mathrm{d} x,
\end{equation}\smallskip
valid in the region $h \in [-\tfrac{Q^2}{2}, \tfrac{Q^2}{2}]$. Here comes the importance of restricting the range of the shift variable $h:= n-m$, we now can choose $Q = \sqrt{N/K}$ (instead of $\sqrt{N}$). The function $g(q,x)$ is bounded and satisfies\smallskip
\begin{equation}
	g(q,x) \ll_A |x|^{-A}
\end{equation}
\vspace{-0.2cm}

\noindent for any $A>1$, see \cite[(6)]{munshi_subconvexity_2018}. So for our purposes we can consider the $x$-integral essentially in the bounded interval $|x|\ll_\varepsilon 1$.
The sums and integrals in \eqref{eq:S(N)-original} can be interchanged and\smallskip
\begin{equation}
\label{eq:summary1}
S(N) = \frac{1}{KQ} \int_{\R} \int_{\R} V\left(\frac{v}{K} \right) \sum_{q \leqslant Q} \frac{g(x,q)}{q} \sums_{a \text{ mod } q}  \mathcal{M}(a,q,x,t + v) \mathcal{N}(a,q,x,v)  \d x \d v .
\end{equation} \smallskip

Here $\mathcal{M} = \mathcal{M}(a,q,x,t+v)$ is defined as
\begin{equation}
\label{eq:M-definition}
\mathcal{M} =  \sum_{m=1}^\infty  \chi(m) m^{-i(t + v) } e\left(-\frac{am}{q} - \frac{mx}{qQ}\right) U\left(\frac{m}{N}\right)
\end{equation}\smallskip

and $\mathcal{N} = \mathcal{N}(a,q,x,v)$ is 
\begin{equation}
\label{eq:N-definition}
	\mathcal{N} =  \sum_{n =1}^\infty \lambda(1,n) n^{iv} e\left(\frac{an}{q}+ \frac{nx}{qQ}\right)  U\left(\frac{n}{N}\right).
\end{equation}\smallskip

The right side of \eqref{eq:delta} is trivially bounded by $O((pt)^\varepsilon)$. So it may look like the loss is not great, but note that now the $m,n$ variables have been separated and so we have an extra sum of length $N$. Thus we need to save $N$ plus a little extra.

\section{Voronoï formulas}

We now apply Poisson and Vorono\"{i} formulas the $\mathcal{M}$ and $\mathcal{N}$ sums.

\subsection{The $\GL_1$ Poisson}
\label{sec:GL1-poisson}

To take advantage of both the modulus $p$ of $\chi$ and the modulus $q$ of $e(\tfrac{a\cdot}{q})$, apply Poisson summation formula after splitting into classes modulo $pq$.

\begin{lemma}
\label{lem:M0bound}
Let
\begin{equation}
\label{I-integral}
 I(m,x)  :=   \int_{\R} \xi^{-i(t + v)} e\left(-\frac{N \xi  m}{pq} - \frac{N \xi x}{qQ}\right) U(\xi) \d \xi.
\end{equation}
We have that
\begin{equation}
\label{eq:GL1-poisson}
\mathcal{M}= N^{1 - i(t +v)}\frac{\tau(\chi)}{p} \sum_{\substack{ m \ll M_0 \\ m \equiv ap \mod q}} \psi(m,q,a) I(m, x) + O_A((pt)^{-A}),
\end{equation}
where $\tau(\chi)$ is the Gauss sum of $\chi$ and where we used the notation
\begin{equation}
	\psi(m,q,a):= \begin{cases} \chi(q) \overline{\chi(m)} & \text{ if } (p,q) = 1,\\[0.3em]  \chi\big(\frac{q}{p^\ell}\big)  \overline{\chi\big(\frac{(m - ap)}{p^\ell}\big)}  &\text{ if } p^\ell \| q, \text{ with } \ell \geq 1. \end{cases}
\end{equation}

For the behaviour of $I(m,x)$ there are two regimes depending on the size of~$q$. Put
\begin{equation}
	M_0 = M_0(q) := \begin{cases} \displaystyle \frac{pqt}{N} & \text{ if } q\gg t^\delta \sqrt{NK}/t,\\[.7em] \displaystyle \frac{p\sqrt{K}}{\sqrt{N}} & \text{ otherwise.} 
	\end{cases}
\end{equation}
If $q \gg t^\delta \sqrt{NK}/t$ (for some fixed $\delta>0$) then $I(m,x) \ll_{A} (pt)^{-A}$ for any $A>0$ and uniformly in $x\in \R$, except if $|m| \asymp_\varepsilon M_0.$  If $q \ll t^\delta \sqrt{NK}/t$ then $m$ is restricted to $|m| \ll M_0$. 

In either case $I(m,x) \ll t^{-1/2}$.
\end{lemma}

\begin{proof}
Applying Poisson summation modulo $pq$ we get,
\begin{equation}
	\mathcal{M} = \frac{N^{1-i(t + v)}}{pq} \sum_{m \in\Z} \sum_{u \pmod{pq}} \chi(u) e\left(-\frac{au}{q} + \frac{mu}{pq} \right) I(m,x).
\end{equation}

If $(p,q) = 1$ we can apply Chinese Remainder Theorem and factor the arithmetic sum modulo~$pq$, so that we get
\begin{align}
	 \sum_{u \pmod{pq}} \chi(u) e\left(-\frac{au}{q} + \frac{mu}{pq} \right)  &=  \sum_{u_1 \pmod{p}} \chi(u_1) e\left(\frac{m\overline{q}u_1}{p} \right) \sum_{u_2 \pmod{q}} e\left(\frac{m\overline{p} - a}{q}u\right)\\
	 & = \overline{\chi(m\overline{q}) }\tau(\chi) q \delta_{m \equiv pa \pmod{q}}.
\end{align}

If $q = p^\ell q'$ with $(q', p) = 1$ and $\ell\geq 1$, we factorize the sum into prime powers. For primes different than $q$, we only get the condition $m \equiv ap \pmod{q'}$. For the $p$-factor, write $u = u_1 + pu_2$ where $u_1$ runs modulo $p$ and $u_2$ runs modulo $p^\ell$. The sum becomes,
\begin{align}
	 \sum_{u \pmod{pq}} \chi(u) e\left(-\frac{au}{q} + \frac{mu}{pq} \right) &= q' \delta_{m \equiv ap \pmod{q'} } \sum_{u \pmod{p^{\ell+1}}} \chi(u) e\left(\frac{-au\overline{q'}}{p^\ell} + \frac{m u\overline{q'}}{p^{\ell+1}}\right)   \\ \notag
	 &= q' \chi(q') \sum_{u_1 \pmod{p}} \chi(u_1) e\left(\frac{(m-ap)/p^\ell}{p}u_1 \right)  \sum_{u_2 \pmod{p}} e\left(\frac{m- ap}{p^\ell}u_2\right).
\end{align}
The second sum is $p^\ell \delta_{m \equiv ap \pmod{p^\ell}}$, thus in the first sum $(m-ap)/p^\ell$ is an integer, and using properties of Gauss sums we obtain $\psi(m,q,a)$.

For the $I(m, x)$ integral, we apply the stationary phase argument as in \cite{kiral2019oscillatory}. Introduce the phase~$\phi(\xi) = -\frac{1}{2\pi}(t+v) \log(\xi) - \frac{Nm \xi}{pq} - \frac{Nx\xi}{qQ}$ so that we have $I(m,x) = \int_{\R} U(\xi) e(\phi(\xi))\d \xi.$ Recall that $U$ is supported in $[1,2]$. We have 
\begin{equation}
	\phi'(\xi) = -\frac{t + v}{2\pi \xi } - \frac{Nm}{pq} - \frac{Nx}{qQ}.
\end{equation}

Since $K < t^{1-\varepsilon}$ we have  $t + v \asymp t$. In the large $q$ regime, $Nx/qQ \ll t^{1-\delta}$, therefore the stationary point would fall inside the support of $U$ only if $m \asymp_\varepsilon M_0$.  If $q \asymp \sqrt{NK}/t$, then $Nm/pq$ can go up to size $t$ without harming the stationary point. That is $|m| \ll pqt/N \asymp M_0$. If $q \ll t^\delta \sqrt{NK}/t$, then for a stationary point to occur the $Nm/pq$ term and the $Nx/qQ$ term should be of the same size, i.e. $|m| \asymp M_0$. In these cases where the stationary point is inside the support of $U$, the second derivative bound gives $I(m,x)\ll  t^{-1/2}$.

Now for the constant term, if $(p,q) = 1$ then $\chi(m)$ gets rid of the $m=0$ term. Also, if $p^2 | q$, then $ap$ is never congruent to zero modulo $q$. So we only have the case $p\|q$.  But also then $\phi'(\xi) = -\frac{t + v}{2\pi \xi} - \frac{Nx}{qQ}$. Thus since
\begin{equation}
	\frac{Nx}{qQ} \ll \frac{N}{pQ} \ll \frac{\sqrt{NK}}{p} \ll t^{1-\varepsilon},
\end{equation}
we know that the stationary point is outside of the support of $U$.
\end{proof}

We can now bound roughly
\begin{equation}\label{eq:Mroughbound}
\sums_{a \mod{q}} \mathcal{M}(a,q,x, t + v) \ll \frac{NM_0}{\sqrt{pt}} \ll_\varepsilon q\sqrt{pt}.
\end{equation}

At this point, just for tracking our progress, if we apply a trivial bound on \eqref{eq:summary1}, we get 
\begin{align}
	S(N) & \ll \frac{1}{KQ} \int_\R \int_\R V\left(\frac{v}{K}\right) \sum_{q \leq Q} \frac{g(x,q)}{q} \left(\ \sums_{a \pmod{q}} \mathcal{M}(a,q)\right) \max_{a\pmod{q}^*} \mathcal{N}(a,q)\d x \d v \\
	& \ll_\varepsilon\frac{1}{Q} \sum_{q \leq Q} \int_{\R} g(x,q) \d x \sqrt{pt} N \ll_\varepsilon \sqrt{pt}N \ll_\varepsilon \sqrt{N} (pt)^{3/4 + 1/2},
\end{align}
where we used the Ramanujan-on-average bound \cite[(2.6)]{li_bounds_2011} on the $\GL_3$ Fourier coefficients after Cauchy-Schwarz for the $\mathcal{N} \ll_\varepsilon N$ bound. Notice that by this application of Poisson, we gained our foothold back from a lost $N= (pt)^{3/2}$ position to a lost $(pt)^{1/2}$ from convexity.

\subsection{The $\GL_3$ Vorono\"{i}}
\label{sec:GL3-poisson}

The $\GL_3$ Vorono\"{i} summation  \cite{miller2006automorphic}, see also \cite[Lemma 3]{blomer_subconvexity_2012}, reads
\begin{equation}
\label{eq:voronoi-GL3}
	\sum_{n = 1}^\infty \lambda(1,n) e\left(\frac{na}{q}\right) g(n) = q \sum_{\pm} \sum_{n_1 | q} \sum_{n_2 = 1}^\infty \frac{\lambda(n_2,n_1)}{n_1n_2} S\left({\overline{a}}, \pm n_2, \tfrac{q}{n_1}\right)   G_\pm \left(\frac{n_1^2 n_2}{q^3}\right).
	\end{equation}
Here we introduce the Mellin-Barnes integral, for $\ell \in \{0, 1\}$, 
\begin{equation} \label{eq:Gdefn}
	G_{\ell}(z) = {\frac{\pi^{3/2}}{2}}\frac{1}{2\pi i } \int_{(\sigma) } (\pi^3 z)^{-s} \frac{\Gamma(\frac{1+s + \alpha_1 + \ell}{2})\Gamma(\frac{1+s + \alpha_2 + \ell}{2})\Gamma(\frac{1+s + \alpha_3 + \ell}{2})}{\Gamma(\frac{-s - \alpha_1 + \ell}{2}) \Gamma(\frac{-s - \alpha_2 + \ell}{2}) \Gamma(\frac{-s - \alpha_3 + \ell}{2})} \widetilde{g}(-s) \d s, 
\end{equation}
\noindent and define $G_{\pm} = G_0 \pm i G_1$. In the above, $(\alpha_1, \alpha_2, \alpha_3)$ are the Langlands parameters of $\pi$ and the function $\widetilde{g}$ is the Mellin transform of $g$ defined by
\begin{equation}
\widetilde{g}(s) = \int_0^\infty g(y) y^{s-1} \d y.
\end{equation}

We work with $g(y) = U(\tfrac{y}{N}) y^{iv}  e(\tfrac{xy}{qQ})$ from the $\mathcal{N}$ sum \eqref{eq:N-definition}. Also from now on we will focus on $G = G_+$ as the \emph{minus case} is treated \emph{mutatis mutandis}.

The function $g$ is supported in $[N, 2N]$, then in the range $yN \gg N^\varepsilon$ we can extract the modulus and phase of $G(y)$ explicitly \cite[Lemma 2.1]{li_bounds_2011}. This motivates to separate the treatment of \eqref{eq:voronoi-GL3} into two cases: the complementary range $\tfrac{n_1^2n_2}{q^3} \ll \tfrac{N^\varepsilon}{N}$ and the main range $\tfrac{n_1^2n_2}{q^3} \gg  \tfrac{N^\varepsilon}{N}$. Let us call $\Nmain$ the contribution from $n_1^2n_2 \geq \tfrac{N^\varepsilon q^3}{N}$ terms in \eqref{eq:voronoi-GL3}, and $\Ncomp$ the remaining sum. The decomposition $\mathcal{N} = \Nmain + \Ncomp$ also gives us a decomposition $S(N) = S_{\text{main}}(N) + S_{\text{comp}}(N)$ via \eqref{eq:summary1}.

\subsubsection{The complementary range}

In the complementary range, we follow \cite[Section 3.3]{munshi_subconvexity_2018}.

\begin{lemma}
In the range $z \ll_\varepsilon  \tfrac{1}{N}$ and $G$ as in \eqref{eq:Gdefn} we have
\begin{equation}
\label{eq:G-bound}
G\left( z \right) \ll_\varepsilon  K^{-1/2} .
\end{equation}
\end{lemma}
\begin{proof}
Firstly note that 
\begin{equation}
\widetilde{g}(-s) = \int_{0}^\infty g(y) y^{-s} \frac{\d y}{y} = N^{-s+iv} \int_0^\infty U(y) \eta^{-y} e^{i \left(2\pi \frac{xN}{qQ}y + (v - \tau)\log y\right)} \frac{\d y }{y},
\end{equation}
with $s = \sigma + i\tau$. Secondly note that from stationary phase methods the integral is negligible unless 
\begin{equation}
\left|2\pi \frac{xN}{qQ} + \frac{v - \tau}{y} \right| \ll_\varepsilon 1
\end{equation}
which translates to $|v - \tau| \asymp \frac{xN}{qQ}$. Then the second derivative yields $\widetilde{g}(-s) \ll N^{- \sigma } |\tau  - v|^{-1/2}.$

Now in \eqref{eq:Gdefn} the ratio of Gamma functions are approximated as $(1 + |\tau|)^{\frac{3}{2} + 3\sigma}$ using the fact that $\alpha_1 + \alpha_2 + \alpha_3 = 0$. 
Move the line of integration to $\sigma = -3/2$. For $\ell = 1$ the integrand is analytic on the region in between since $|\Re(\alpha_i) | \leq 1/2$. For $\ell =0$ we pass three poles at $s =-1 -\alpha_i$, $i \in \{ 1,2,3\}$. At that vertical line of integration, we have that the integral is supported essentially for $|v - \tau| \ll_\varepsilon xN/qQ.$ We now bound
\begin{equation}
	G_\ell (z) \ll \int_{(-3/2)} z^{-\sigma} (1 + |\tau|)^{-3} |\widetilde{g}(-s)| \d s  +\delta_{\ell = 0} \sum_{i =1}^3 (\pi^3 z)^{1+\alpha_i }\frac{\prod_{j \neq i} \Gamma(\frac{\alpha_j - \alpha_i}{2})}{\prod_{j = 1}^3 \Gamma( \frac{1 + \alpha_i - \alpha_j}{2})} \widetilde{g}(1 + \alpha_i).
\end{equation}
The contribution of the residues is bounded by $z^{1+\alpha_i} |\widetilde{g}(1 + \alpha_i)| \ll_\varepsilon |\mathrm{Im}(\alpha_i) - v|^{-1/2} \ll_\varepsilon\tfrac{1}{K^{1/2}}$ in the complementary range.

For the integral we make a change of variables $\tau \mapsto \tau + v$ and separate it into the regions $|\tau| < 1$ and $1\leq |\tau| \ll xN/qQ$. So the integral we have is bounded by
\begin{equation}
	\int_{|\tau|\leq 1} |\tau + v|^{-3} |\tau|^{-1/2} \d \tau + \int_{1}^{\frac{xN}{qQ}} (1 + |\tau + v|)^{-3} \d \tau + \int_1^{\frac{xN}{qQ}} (1 + |v- \tau |)^{-3} |\tau|^{-1/2} \d \tau.
\end{equation}
In the first integral $|\tau|^{-1/2}$ is integrable so its contribution is $K^{-3}$. For the second integral we have decided to ignore $|\tau|^{-1/2} \leq 1$ term and bound it simply by $K^{-2}$. For the third integral, splitting the integral dyadically $\frac{K} {2^{i + 1}}  \leq  |v - \tau| \leq \frac{K}{2^{i}}$ we obtain the result.
\end{proof}

Using the bound \eqref{eq:G-bound}  and the Weyl bound for Kloosterman sums in \eqref{eq:voronoi-GL3}, we get a bound for the $\GL_3$-sum in the complementary range given by
\begin{equation}
\Ncomp \ll_\varepsilon \frac{q^{3/2}}{K^{1/2}} \sum_{\substack{n_1 \mid q, n_2\\ n_1^2 n_2 \ll  q^3/N }} \frac{|\lambda(n_2, n_1)|}{n_1^{3/2}n_2}.
\end{equation}

We apply Cauchy inequality and get
\begin{equation}
\Ncomp \ll_\varepsilon \frac{q^{3/2}}{K^{1/2}} \left( \sum_{n_1^2n_2 \ll q^3  / N} \frac{1}{n_1n_2} \right)^{1/2} \left( \sum_{n_1^2n_2 \ll  q^3  / N} \frac{|\lambda(n_2, n_1)|^2}{n_1^2n_2} \right)^{1/2}.
\end{equation}
The first term is essentially bounded. We appeal to the Ramanujan  bound on average \cite[(2.6)]{li_bounds_2011}, so that the second term is also essentially bounded. We therefore get 
\begin{equation} \label{eq:Ncompbound}
\Ncomp \ll_\varepsilon \frac{q^{3/2}}{K^{1/2}} \ll_\varepsilon \frac{N^{3/4}}{K^{5/4}}.
\end{equation}

Putting together \eqref{eq:Mroughbound} and \eqref{eq:Ncompbound} inside \eqref{eq:summary1}, we deduce
\begin{align}
\label{a}
\Scomp \ll_\varepsilon \frac{\sqrt{pt} N^{3/4}}{K^{5/4}} =  \frac{N^{3/4}\sqrt{pt}}{K^{5/4}} { = N^{1/2}(pt)^{3/4} \cdot \frac{ N^{1/4}}{(pt)^{1/4}K^{5/4}}}.
\end{align}
\noindent This is a subconvex bound for any value of $K > (pt)^{1/10}$.

\subsubsection{The main range}

In this range the weight function $G$ behaves as follows. Since $g$ is supported in $[N, 2N]$, then in the range $zN \gg N^\varepsilon$ we can extract the modulus and phase of $G_{\pm}(z)$ \cite[Lemma 2.1]{li_bounds_2011} and write for any $A>0$,
\begin{equation}
G_{\pm}(z) =  {\alpha} z^{2/3} \int_\R  g(y) y^{-1/3} e\left(
\pm z^{1/3}y^{1/3}\right)d y + O_A((zN)^{-A}).
\end{equation} 
{for a certain constant $\alpha$, depending only on $\pi$.}
Substituting $z = n_1^2n_2/q^3$, $g(y) = y^{iv}e\big(\frac{xy}{qQ}\big)U\left(\frac{y}{N}\right)$ and the asymptotic expansion above into \eqref{eq:voronoi-GL3} and changing the variable $y \mapsto Ny$, we obtain
\begin{align}
\label{GL3-voronoi}
 \Nmain =&  \alpha \frac{N^{2/3+iv}}{q} \sum_{\substack{n_1\mid q, n_2\\ n_1^2 n_2 \gg q^3/N}} n_1^{1/3}n_2^{-1/3} \lambda(n_2, n_1) S\left({\overline{a}}, n_2, \frac{q}{n_1}\right)  \\
&  \quad \times \int_\R y^{iv-1/3} U(y) e\left( \frac{Nyx}{qQ} + \frac{(Nn_1^2n_2y)^{1/3}}{q} \right)  \dd y.\notag
\end{align}

Integrating by parts repeatedly the integral appearing above, we get a term majorized by
\begin{equation}
\left( \frac{q}{(Nn_1^2n_2)^{1/3}} \right)^k \left[ v^k + \left( \frac{Nx}{qQ} \right)^k \right]
\end{equation}

\noindent for all $k \geqslant 0$, so that the integral is vanishingly small except when
\begin{equation}
n_1^2n_2 \ll_\varepsilon \frac{(qK)^3}{N} + K^{3/2}N^{1/2} \ll_\varepsilon K^{3/2}N^{1/2} =: N_0.
\end{equation}

Let us summarize: $\Smain$ is equal to
\begin{align}
\notag  &  \alpha \frac{N^{5/3{-it}}{\tau(\chi)}}{{p} KQ} \sum_{q \leqslant Q} \frac{{1}}{q^2} {{\sums_{a \mod q}}}  \sum_{\substack{ m \ll M_0 \\ {m \equiv ap \mod q}}}  {\psi(m, q, a)} \sum_{\substack{n_1 \mid q, n_2\\ q^3/N \ll n_1^2 n_2 \ll N_0}} n_1^{1/3}n_2^{-1/3} \lambda(n_2,n_1) S\left({\overline{a}}, n_2, \frac{q}{n_1}\right)\\
&\quad \times \int_\R \int_\R V\left( \frac{v}{K} \right)  g(x,q) \int_\R U(\xi) \xi^{-i(t+v)} e\left(-\frac{Nm\xi}{pq	} - \frac{N\xi x }{qQ}\right) \d\xi  \\
&\quad \times \int_{\R} y^{iv - 1/3} U(y) e\left( \frac{Nyx}{qQ} + \frac{(Nn_1^2n_2 y)^{1/3}}{q} \right)\d y  \d x \d v + O_A(N^{-A}),\notag
\end{align}

\noindent where we recall the bounds 
\begin{align}
1 & \ll  m \ll_\varepsilon M_0  \label{eq:M0defn} \\
q^3/N & \ll n = n_1^2n_2 \ll_\varepsilon N_0 = N^{1/2}K^{3/2} \label{eq:N0defn} \\
N/pt & \ll q  \ll N^{1/2}/K^{1/2}. \label{eq:CBound}
\end{align}

\subsection{Simplifying the $x$-integral}
\label{subsec:x-integral}

We now concentrate on the $x$ and the $v$-integrals,
\begin{equation}
\mathcal{W} := \frac{1}{K} \int_{\R} V\left(\frac{v}{K} \right) \int_\R g(x,q) e\left( \frac{Nx}{qQ}(y-\xi) \right)  \mathrm{d} x  \left(\frac{y}{\xi}\right)^{iv}\d v.
\end{equation}
Using the same arguments as in \cite[\S 4.1]{munshi_subconvexity_2018}, the integral is vanishingly small unless
\begin{equation} \label{eq:usize}
|y-\xi| \ll_\varepsilon \frac{qQ}{N}.
\end{equation}
This motivates the change of variable $y = \xi + u$ with $u \ll_\varepsilon qQ/N$ (note that this is small since $qQ/N \ll 1/K$).  We get
\begin{align}
\notag\Smain & = \alpha \frac{N^{5/3{-it}} {\tau(\chi)}}{{ p} Q} \sum_{q \leqslant Q} \frac{{1}}{q^2} {{\sums_{a \mod q}}} \sum_{\substack{ m \ll M_0\\ {m\equiv ap\pmod{q}}}}  {\psi(m, q,a)} \\
& \qquad \times  \sum_{\substack{n_1 \mid q \\ q^3/N \ll n_1^2n_2\ll N_0}} n_1^{1/3}n_2^{-1/3} \lambda(n_2,n_1) S\left({\overline{a}}, n_2, \frac{q}{n_1}\right) \\
& \notag   \qquad \times  \int_{|u|\ll \frac{qQ}{N}}  I(m,n_1^2 n_2,q,p)\mathcal{W} \d u + O_A(N^{-A}).
\end{align}
Here we introduced
\begin{equation} \label{eq:Idefinition}
I(m, n_1^2n_2, q, p) = \int_0^\infty U_1(\xi) \xi^{-it} e\left( -\frac{Nm\xi}{pq}+ \frac{(Nn_1^2n_2 (\xi + u))^{1/3}}{q} \right) \d \xi.
\end{equation}
where $U_1(\xi) = U(\xi) U(\xi + u) (\xi + u)^{-1/3}$.
It is a nonoscillating bump function, and we may as well drop the subscript from now on.

The $u$ integral is on a small interval, of size $\ll_\varepsilon 1/K$. We bound it trivially by the supremum of its value times the length of the integral.  Furthermore, we cut the $q$ sum dyadically into pieces $q \sim C$. Therefore we have
\begin{equation} \label{SmainBound}
 \Smain \ll_\varepsilon  \sup_{u \ll qQ/N} \sum_{\substack{C\ll Q,\\ C \text{ dyadic}}}\frac{N^{2/3}}{\sqrt{p}} |\Smainprime(N,C)|.
\end{equation}
Here the $u$ and the $v$ dependence is inside $\mathcal{W}$ and $I$, and $\Smainprime(N,C)$ is defined as
\begin{equation} \label{eq:Smainprime}
\sum_{q \sim C} \frac{1}{q}  {{\sums_{a \mod q}}} \sum_{\substack{m \ll M_0\\ {m \equiv ap \mod{q}}}} \!\!\!\!\! {\psi(m, q,a)} \hspace{-0.3cm} \sum_{\substack{n_1 \mid q \\ q^3/N \ll  n_1^2n_2\ll N_0}} \hspace{-0.5cm} n_1^{1/3} \frac{\lambda(n_2,n_1)}{n_2^{1/3}} S\left({\overline{a}}, n_2, \frac{q}{n_1}\right) \mathcal{W} I(m, n_1^2n_2, q, p) .
\end{equation}

\section{Poisson Summation Formula}
\label{sec:cauchy}

We bring the $m$ and $q$ sums inside and take absolute values, thus giving up on obtaining sign cancellation from $\lambda(n_2, n_1)$. We now have, 
\begin{align}
\Smainprime(N,C) \ll \sum_{n_1 \ll C} n_1^{1/3}  \sum_{n_2 \ll N_0/n_1^2} \frac{|\lambda(n_2, n_1)|}{n_2^{1/3}} \Bigg| \sum_{\substack{q \sim C\\ n_1 | q }} {{\sums_{a \mod q}}} \sum_{\substack{m \ll M_0 \\ m \equiv ap \mod{q}}} \frac{1}{q} { \psi(m, q,a)}\mathcal{C} \mathcal{I} \Bigg|
\end{align}
\noindent where
\begin{align}\label{eq:CandI}
\mathcal{C} & = S({\overline{a}}, n_2, q/n_1), \\
\mathcal{I}  & = \mathcal{W} I(m, n_1^2n_2, q, p).
\end{align}

We can therefore use Cauchy-Schwarz inequality on the $n_2$-sum to write
\begin{equation}
\label{eq:SmainPrimeCauchySchwartz}
\Smainprime \ll  \sum_{{n_1 \ll C}} n_1^{1/3} \Theta^{1/2} \Omega^{1/2}
\end{equation}
\noindent where
\begin{equation}
\Theta = \sum_{n_2 \ll N_0/n_1^2} \frac{|\lambda(n_2, n_1)|^2}{n_2^{2/3}}
 \quad \text{and} \quad
\Omega = \sum_{n_2 \ll N_0/n_1^2} \Bigg| \sum_{\substack{q \sim C \\ n_1 \mid q}} {{\sums_{a \mod q}}} \sum_{\substack{m\ll M_0\\ m \equiv ap \mod{q}}}\frac{1}{q} { \psi(m,q,a)} \mathcal{C}  \mathcal{I} \Bigg|^2.
\end{equation}

Expanding the square in $\Omega$ and explicitly denoting the bound on $n_1^2n_2$ by way of a bump function $\varphi$ with $\operatorname{supp} \varphi \subseteq [1,2]$, we can write
\begin{equation} \label{eq:Omega}
\Omega \ll \sum_{\substack{q, q' \sim C\\ n_1 \mid q, q'}} {{\sums_{\substack{a \mod q\\ a' \mod q'}}}} \sum_{\substack{m, m'\ll M_0\\{m \equiv ap \mod q}\\ {m' \equiv a'p \mod q'}}} \sum_{n_2 \ll N_0/n_1^2} \, \frac{1}{qq'}  {\psi(m, q,a) \overline{\psi}(m', q',a')} \mathcal{C}\overline{\mathcal{C}'} \mathcal{I} \overline{\mathcal{I}'} \varphi\left( \frac{n_1^2n_2}{N_0} \right).
\end{equation}
Here $\mathcal{I}'$ or $\mathcal{C}'$ indicate that in \eqref{eq:CandI}, the variables are taken to be primed.

In the next lemma we apply Poisson summation formula on the $n_2$-summation. To take advantage of both moduli $q/n_1$ and $q'/n_1$ in the Kloosterman sums $\mathcal{C}$ and $\mathcal{C}'$ we consider it modulo $B:=qq'/n_1^2$.

\begin{lemma} \label{lem:OmegaBound}
Let ${\Xi = \Xi(q, q', m, m',a,a') = \psi(m, q,a) \overline{\psi}(m', q',a')/qq'}$. With $\Omega$ and $B$ as above, we have
\begin{equation} \label{eq:OmegaBound}
\Omega  \ll \frac{N_0}{n_1^2} \sum_{n_2 \ll N_0/n_1^2} \sum_{\substack{q, q' \sim C\\ n_1 \mid q, q'}} \frac{1}{B} {{\sums_{\substack{a \mod q\\ a' \mod q'}}}} \sum_{\substack{m, m' \ll M_0\\{m \equiv ap \mod q}\\ {m' \equiv a'p \mod q'}}} \Xi \mathfrak{C} \mathfrak{I}.
\end{equation}
Here
\begin{align}
\label{eq:I-frak-integral}
\mathfrak{I} &:= \int_\R \varphi(w) \mathcal{I}(m, N_0 w, q) \overline{\mathcal{I}}(m', N_0 w, q') e\left( -\frac{N_0n_2w}{Bn_1^2} \right) \d w, \\
\mathfrak{C} &:= \sum_{r \pmod B} S({\overline{a}}, r, \tfrac{q}{n_1}) \overline{S({\overline{a'}}, r, \tfrac{q'}{n_1})} e\left( \frac{n_2r}{B} \right).
\end{align}
\end{lemma}

\begin{proof}
Separating the $n_2$ sum in \eqref{eq:Omega} into residue classes $n_2 = r + \ell B$ modulo $B$, 
\begin{align}
\Omega & \ll \sum_{\substack{q, q' \sim C\\ n_1 \mid q, q'}} {{\sums_{\substack{a \mod q\\ a' \mod q'}}}}  \sum_{\substack{m, m' \ll M_0\\ {m \equiv ap \mod q}\\ {m' \equiv a'p \mod q'}}} \Xi   \sum_{r \pmod B} S({\overline{a}}, r, \tfrac{q}{n_1}) \overline{S}({\overline{a'}}, r, \tfrac{q'}{n_1}) \sum_{\ell \in \Z} \varphi\left( \frac{n_1^2(r + \ell B)}{N_0} \right) \mathcal{I} \overline{\mathcal{I}'}.
\end{align}
Recall that $\mathcal{I} = \mathcal{I}(m, n_1^2(r + \ell B), q, p)$ and $\mathcal{I}' = \mathcal{I}(m', n_1^2(r + \ell B), q', p)$. We now apply Poisson summation formula in the $\ell$-sum and write it in the form
\begin{equation}
\sum_{\ell\in\Z} f(r + \ell B) = \sum_{\ell \in\Z} \widetilde{f}(\ell),
\end{equation}

\noindent where
\begin{align}
\widetilde{f}(\ell) & =\int_\R \varphi \left(\frac{n_1^2(r + uB)}{N_0} \right) \mathcal{I}(m, n_1^2(r+ uB), q) \overline{\mathcal{I}}(m', n_1^2(r + uB), q') e(-\ell u) \d u, 
\end{align}
and with the change of variable $w = \frac{r + uB}{N_0}n_1^2$ we have
\begin{equation}
\widetilde{f}(\ell)= \frac{N_0}{B n_1^2} \int_\R \varphi(w) \mathcal{I}(m, N_0 w, q) \overline{\mathcal{I}}(m', N_0 w, q') e\left( \frac{\ell r}{B} - \frac{\ell N_0 w}{B n_1^2} \right) \d w.
\end{equation}

Thus we get \eqref{eq:OmegaBound} after relabelling $\ell$ by $n_2$. 
\end{proof}

\section{Final bounds}
\label{sec:bounds}

\subsection{Bounds on the integrals}
\label{sec:analytic}

We use the oscillations of the $\mathfrak{I}$-integral in order to get a bound on the $n_2$-range and to bound $\mathfrak{I}$. For that purpose we apply the stationary phase theorem to $\mathcal{I}$, thus obtaining the oscillation in $w$, and then using this information to apply stationary phase to $\mathfrak{I}$.

\begin{lemma}\label{lem:Ibound}
The integral $\mathfrak{I}$ is vanishingly small except in the range $n_2 \ll_\varepsilon N_1 := \tfrac{C}{K}$. Moreover we have $\mathfrak{I} \ll t^{-1}$ and, for $C \gg t^{\delta} \max\{\sqrt{NK}/t, \sqrt{N/Kt}\}$ with any fixed $\delta>0$, we have
\begin{equation} \label{eq:Ibound}
\mathfrak{I} \ll \frac{1}{t} \left(\frac{qq'}{N_0n_2}\right)^{{1/2}}.
\end{equation}

\end{lemma}
\begin{proof}
We follow \cite[Section 4.2]{munshi_subconvexity_2018} and apply the stationary phase bound \cite[Main Theorem]{kiral2019oscillatory}. First let us study the oscillation in the  $I(m,N_0w,q,p)$  integral. Recall that
\begin{equation} 
\label{eq:integral-I-reminder}
I(m,N_0w,q,p) = \int_{\R} U(\xi) e\left(-\frac{t}{2\pi}\log \xi + A\xi + B(\xi + u)^{1/3}\right) \d\xi
\end{equation}
where $A = - Nm/pq$ and $B = (NN_0w)^{1/3}/q$.

Except for when $B \asymp_\varepsilon t$, one can deduce $I\ll_\varepsilon 1/\sqrt{t}$ from the second derivative bound on~\eqref{eq:integral-I-reminder}. But if $B\asymp_\varepsilon t$, we can also apply the argument in \cite[Lemma 1]{munshi_subconvexity_2018}:
\[
	\mathfrak{I} \leq \int_0^\infty \int_0^\infty U(\xi_1) U(\xi_2) \left| \int_\R \varphi(w^3)3w^2 e\left( Bw((\xi_1 + u)^{1/3} - (\xi_2 + u)^{1/3}) \right) \d w\right| \d \xi_1 \d \xi_2.
\]
When $B\asymp_\varepsilon t$, repeated integration by parts gives $|\xi_1 -\xi_2| \ll 1/t$ on the inner integral. This proves the $\mathfrak{I}\ll t^{-1}$ bound. {We also have $\mathfrak{I}$ essentially supported in $n_2 \ll N_1$. This is because if we open up $\mathcal{I}$ and $\mathcal{I}'$ and focus on the oscillation in the $w$-variable we get
\begin{equation}
e\left( \frac{(NN_0(\xi + u))^{1/3}}{q}w^{1/3} + \frac{(NN_0(\xi' + u))^{1/3}}{q'}w^{1/3} - \frac{n_2N_0}{qq'} w\right).
\end{equation}
The stationary point of this oscillation is unbounded if $n_2N_0/qq' \gg_\varepsilon (NN_0)^{1/3}/C$ so that $\mathfrak{I}$ is negligible except when 
\begin{equation}
n_2 \ll_\varepsilon \frac{N^{1/3}C}{N_0^{2/3}} = \frac{C}{K} = N_1.
\end{equation}
}

From now on assume that $C\gg t^\delta \sqrt{NK}/t$, thus $B \ll t^{1-\delta}$. Because we are in this large modulus regime, from Lemma \ref{lem:M0bound}, we have that $|m|\asymp_\varepsilon M_0$, and thus $A \asymp_\varepsilon t$.

Secondly note that since $|u|\ll_\varepsilon 1/K$ via \eqref{eq:usize} we can appeal to the binomial expansion $(\xi + u)^{1/3} =  \xi^{1/3} +  \xi^{1/3} \sum_{n=1}^\infty   (1/3)_n (u/\xi)^n/n!$ where $(\alpha)_n = \alpha  (\alpha-1) \cdots (\alpha - n+1)$ is the falling factorial. All the terms corresponding to $n \geqslant 1$ are bounded by $u$, hence the contribution of those terms to the whole phase inside \eqref{eq:integral-I-reminder} is bounded by
\begin{equation}
Bu \ll_\varepsilon \frac{B}{K} \ll_\varepsilon  t^{1/2 - \delta}.
\end{equation}
Here we used that $C \gg t^{\delta} \sqrt{N/Kt}$.
Adding these oscillations to $U$ we get a set of $t^{1/2-\delta} $-inert weight functions $U'$. Since $t^{1/2 - \delta} \ll \sqrt{t} \asymp \sqrt{\phi''(\xi)}$ as we see below in~\eqref{eq:phidoubleprimebound}. By the \emph{Main Theorem} \cite{kiral2019oscillatory}, we may afford to remove all the $u$ from the computations.

From the bound $m \asymp_\varepsilon M_0$ obtained from Lemma \ref{lem:M0bound}, we have 
\begin{equation}
A = \frac{N|m|}{pq}  \asymp_\varepsilon t.
\end{equation}

\noindent In particular, we deduce that $B/A \to 0$ when $t \to +\infty$. For each $t$, the stationary point in~\eqref{eq:integral-I-reminder} is $\xi = \xi_t$ such that
\begin{equation}
\phi'(\xi_t) =  - \frac{t}{2\pi \xi_t} + A + \frac{B}{3}\xi_t^{-2/3} = 0.
\end{equation}
We can also write it as $\xi + (B/3A) \xi^{1/3} - t/2\pi A = 0$, and view this as a perturbed cubic equation, with $\epsilon := B/3A$. Note that $\epsilon \to 0$ as $t \to \infty$. We assume the solution is a power series $\xi = \eta_0 + \eta_1 \epsilon + \eta_2 \epsilon^2+ \cdots$ and solve for the coefficients $\eta_i$. We get that
\begin{equation}
	\xi_t= \frac{t}{2\pi A} - \left(\frac{t}{2\pi A}\right)^{1/3} \epsilon + O(\epsilon^{2}).
\end{equation}

We can apply the stationary phase method to obtain a description of $I(m, N_0w, q)$ with explicit phase and modulus. The modulus is controlled by the second derivative of the phase, given by
\begin{equation}\label{eq:phidoubleprimebound}
\phi''(\xi_t) = \frac{t}{2\pi \xi_t^2} - \frac{2B}{9}\xi_t^{-5/3} \asymp t, 
\end{equation}
since $B=o(t)$. The phase is given by $\phi(\xi_t)$, and expanding the cube root term binomially we get
\begin{equation}
e(\phi(\xi_t)) = \xi_t^{-it} e\left( \frac{t}{2\pi} + \frac{2B}{3}\left(\frac{t}{2\pi A}\right)^{1/3} + O\left( \frac{B^2}{A} \right) \right).
\end{equation}
The stationary phase method hence gives
\begin{equation}
I(m, N_0w, q) \approx \frac{1}{t^{1/2}} \left(\frac{t}{2\pi A}\right)^{-it} e\left(\frac{t}{2\pi} + B \left(\frac{t}{2\pi A}\right)^{1/3} +  O\left( \frac{B^2}{A} \right) \right),
\end{equation}
up to an error term of arbitrary polynomial decay in $pt$.
The $w$ dependence is only in $B$ and pulling the $w$-integral inside, we measure how much we can gain from stationary phase.
Note that we have used $\xi_t^{-it} = \left(\frac{t}{2\pi A}\right)^{-it}e(-\frac{t}{2\pi} \log(1 - (\frac{t}{2\pi A})^{-2/3}\epsilon + O(B^2/A^2)))$.

Denote $\gamma = \tfrac{N_0 n_2}{qq'}$. The inner $w$ integral in \eqref{eq:I-frak-integral} is rewritten as
\begin{equation}
\label{eq:I-frak-unwound}
\mathfrak{I} \approx \frac{1}{t}\int_{\mathbb{R}} \varphi(w) e\left( \frac{(NN_0)^{1/3}}{C}h w^{1/3} + O\left( \frac{B^2}{A} \right) - \gamma w \right) \mathrm{d}w
\end{equation}
where, letting $A' = \tfrac{Nm'}{pq'}$, we defined $h = \frac{C}{q}(\tfrac{t}{2\pi A})^{1/3} - \frac{C}{q'}(\tfrac{t}{2\pi A'})^{1/3}$ which does not depend on $w$ and $h\ll 1$.
We can attain the bound~\eqref{eq:Ibound} by applying the second derivative bound to~\eqref{eq:I-frak-unwound}. Changing variable $w \mapsto w^3$, we have 
\begin{equation}
\label{eq:I-frak-unwound-2}
\mathfrak{I} \approx \frac{1}{t}\int_{\mathbb{R}} 3w^2 \varphi(w^2) e\left( \frac{(NN_0)^{1/3}}{C}hw  + O\left( \frac{B^2}{A} \right) - \gamma w^3 \right) \mathrm{d}w.
\end{equation}
Note that the $w$ dependence in the big-oh term is a power series in $w$ starting from $w^2$. Also since $B/A\ll t^{-\delta}$ the big-oh term is smaller than the first term. Thus, the stationary point of this oscillatory integral is in the support of $\varphi$ only if
\begin{equation}
	\frac{(NN_0)^{1/3}}{C} \asymp_\varepsilon \gamma,
\end{equation}
which translates to $n_2 \asymp_\varepsilon N_1$. By the stationary phase method $\mathfrak{I}$ is negligibly small otherwise. The first term in the exponential is killed in the second derivative and thus the size of the second derivative is $\gamma$. Which means the stationary phase method saves a factor of $\gamma^{-1/2}$.
\end{proof}

\subsection{Bound on Kloosterman sums}
\label{sec:arithmetic}

We now bound $\mathfrak{C}$. These sums have been treated in a previous paper of Munshi \cite[Lemma 11]{munshi_circle_2014} and bounded explicitly in an elementary way. Precisely we have the following: for $n_2 = 0$ we have $\mathfrak{C} = 0$ unless $q=q'$ and then 
\begin{equation} \label{eq:CzeroBound}
\mathfrak{C} \ll B (\tfrac{q}{n_1}, a-a').
\end{equation}
If $n_2 \neq 0$, we have
\begin{equation} \label{CNonzeroBound}
\mathfrak{C} \ll B (\tfrac{q}{n_1}, \tfrac{q'}{n_1}, n_2).
\end{equation}

\subsection{Bounding $\Omega$} \label{subsec:BoundingOmegaa}

Let us write 
\begin{equation}\label{eq:OmegaDecomposition}
	\Omega = \Omega_{0=} + \Omega_{0\neq} + \Omega_{\text{non}}.
\end{equation}
where the terms on the right correspond respectively to the subsum of  \eqref{eq:OmegaBound} corresponding to the indices $n_2=0$ and $a=a'$, $n_2 = 0$ and $a \neq a'$, and finally to nonzero $n_2 \neq 0$. We also put $\Omega_0 = \Omega_{0=} + \Omega_{0\neq}$ accordingly. Since $\Omega^{1/2} \ll \Omega_{0=}^{1/2} + \Omega_{0\neq}^{1/2} + \Omega_{\text{non}}^{1/2}$ we may split $\Smainprime$ in~\eqref{eq:SmainPrimeCauchySchwartz} accordingly, to obtain $\Smainprime \ll S_{\text{main},0=}' + S_{\text{main},0\neq}' + S_{\text{main}, \text{non}}$, where
\begin{equation}
	S_{\text{main}, \square}' = \sum_{n_1\ll C} n_1^{1/3} \Theta^{1/2} \Omega_{\square}^{1/2}.
\end{equation}

We will treat each case separately. The following lemma will be used to show that all the~$n_1$ sums appearing are uniformly bounded. 
\begin{lemma}
\label{lem:n1-sum}
	Let $\alpha \geq 7/6$ be an exponent. Then 
	\begin{equation}
		\sum_{n_1 \ll C} \frac{\Theta^{\frac{1}{2}}}{n_1^\alpha}\ll_\varepsilon N_0^{1/6}.
	\end{equation}
\end{lemma}
\begin{proof}
	We apply Cauchy-Schwarz inequality in the $n_1$ sum, separating $n_1^{-\alpha} = n_1^{-2/3} n_1^{-\alpha + 2/3}$. Note that for $\alpha = 7/6$ the second factor is $n_1^{-1/2}$,
\begin{equation}
\sum_{n_1 \ll C} \frac{\Theta^{\frac{1}{2}}}{n_1^\alpha} \ll \left(\sum_{n_1\ll C} \frac{1}{n_1^{2\alpha - 4/3}} \right)^{1/2}  \left(\sum_{n_1^2 n_2\ll N_0} \frac{|\lambda(n_2,n_1)|^2}{(n_1^2n_2)^{2/3}}\right)^{1/2}.
\end{equation}
The first factor is either a part of a convergent sum if $\alpha > 7/6$ or is the harmonic sum when we have $\alpha = 7/6$ and can be bounded by $C^\varepsilon \ll Q^\varepsilon \ll_\varepsilon 1$.

For the second factor we apply summation by parts. Put
\begin{equation}
A(n) = \sum_{n_1^2n_2 \leq n} \frac{|\lambda(n_2, n_1)|^2}{n_1^2 n_2}.
\end{equation}
By the Ramanujan bound on average, we have that, $A(n) \ll n^\varepsilon.$ Now the second factor can be rewritten using integration by parts as
\begin{equation}
	A(N_0) N_0^{1/3} +  \sum_{n=1}^{N_0-1} A(n) (n^{1/3} - (n+1)^{1/3}) \ll N_0^{\frac{1}{3} + \varepsilon} + \sum_{n = 1}^{N_0-1} \frac{n^{\varepsilon} }{n^{2/3}} \ll_\varepsilon N_0^{\frac{1}{3}}.
\end{equation}
This proves the result.
\end{proof}

First consider the zero frequency case $n_2=0$. The value of $\mathfrak{C}$ is nonzero only if $q=q'$ by the above. Inputting the bound \eqref{eq:CzeroBound} on $\mathfrak{C}$, thus bounding the arithmetic sum $|\Xi|$ by $1/q^2$ and the integral $\mathfrak{I}$ by \eqref{eq:Ibound} we get 
\begin{equation} \label{eq:OmegaZeroBound}
\Omega_0  \ll \frac{N_0}{n_1^2 t}  \sum_{\substack{q \sim C\\ n_1|q}} \frac{1}{q^2} {{\sums_{\substack{a \mod q \\ a' \mod q'}}}} \sum_{\substack{m, m' \ll M_0 \\ {m \equiv ap \mod q} \\ {m' \equiv a'p \mod q'}}}  \left(\frac{q}{n_1}, a-a'\right) .
\end{equation}

\subsubsection{Zero frequency, $a=a'$ case}
 \label{sec:zero-frequency}

Note that $M_0 < q$ in both the small and the large $q$ regimes. We have this bound since the case $N< (pt)^{4/3}$ need not be considered as $S(N)$ can be bounded from the onset via \eqref{eq:SNdefinition} succesfully. Therefore when $a=a'$, we have that $m \equiv m' \mod q$ and this implies $m=m'$. Moreover we have the equality $(\tfrac{q}{n_1}, a-a')~= \tfrac{q}{n_1}$ when $a = a'$. Altogether we can bound
\begin{equation}\label{eq:OmegaZeroEqBound}
\Omega_{0=} \ll \frac{N_0}{n_1^2 t} \sum_{\substack{q\sim C\\ n_1|q}} \frac{1}{qn_1} {{\sums_{\substack{a \mod q}}}} \ \sum_{\substack{m\ll M_0 \\ {m \equiv ap \mod q} }} 1  \ll_\varepsilon \frac{N_0}{t n_1^4} M_0 .
\end{equation}
Assume $C \gg t^\delta \sqrt{NK}/t$. In this case $M_0 \asymp_\varepsilon \frac{Cpt}{N} $. This means that \eqref{eq:OmegaZeroEqBound} can be bounded by $ \ll_\varepsilon \frac{Kp}{n_1^4}$. In the small $q$ regime, as $M_0 \asymp_\varepsilon \frac{p \sqrt{K}}{\sqrt{N}}$ the bound is $\ll_\varepsilon \frac{K^2p}{n_1^4 t}$, which is strictly less than the first case as $K \ll t^{1- \varepsilon}$.

After inputing the $\Omega_{0=}^{1/2} \ll_\varepsilon K^{1/2} p^{1/2}/n_1^2$ bound in \eqref{eq:SmainPrimeCauchySchwartz}, applying Lemma \ref{lem:n1-sum} to deal with the $n_1$ sum, and using~\eqref{SmainBound}, we obtain that the corresponding contribution in $S_\mathrm{main}$ is
\begin{equation}
\label{eq:contribution-0equal}
S_{\mathrm{main}, 0=} \ll_\varepsilon \frac{N^{2/3}N_0^{1/6} K^{1/2}p^{1/2}}{p^{1/2}} \ll_\varepsilon K^{3/4}N^{3/4} \ll_\varepsilon N^{1/2}(pt)^{3/4} \cdot \frac{K^{3/4}}{(pt)^{3/8}}.
\end{equation}

\subsubsection{Zero frequency, $a \neq a'$ case}

When $a=a'$ the expected bound on average is given by $(\tfrac{q}{n_1}, a-a') \ll 1$. More precisely,
\begin{equation}
{\sums_{\substack{a \mod q \\ a' \mod q'}}} (q, a-a') \leq \sum_{d | q} \sum_{\substack{0< a, a' \leq q \\ d| a-a'}} d \ll \sum_{d| q} q \frac{q}{d} d \ll q^{2+\varepsilon}.
\end{equation}
Therefore, 
\begin{equation}\label{eq:OmegaZeroNeqBound}
\Omega_{0\neq}  \ll_\varepsilon \frac{N_0}{n_1^2t} \sum_{\substack{q \sim C\\ n_1|q}} \frac{1}{q^2} \sums_{a ,a' \mod q } \sum_{\substack{m, m' \ll M_0 \\ {m \equiv m'\equiv ap \mod q} \\ }}  \left(\frac{q}{n_1}, a-a'\right)  \ll_\varepsilon \frac{N_0}{n_1^3t} \frac{1}{C}  M_0^2 .
\end{equation}
In the large $q$ regime plugging in $M_0 \asymp pCt/N$ yields $\Omega_{0\neq} \ll_\varepsilon \frac{p^2 tK}{Nn_1^3}$. For the small $q$ regime, we have the bound $\Omega_{0\neq} \ll_\varepsilon\frac{K^{5/2}p^3}{N^{3/2} n_1^3}$ for the same quantity.
For the large $q$ the corresponding contribution in $S_{\mathrm{main}}$ is 
\begin{equation}
\label{eq:contribution-diff}
S_{\mathrm{main}, 0\neq} \ll_\varepsilon \frac{N^{2/3}N_0^{1/6}pt^{1/2}K^{1/2}}{p^{1/2}N^{1/2}} \ll_\varepsilon N^{1/4} K^{3/4} (pt)^{1/2} \ll_\varepsilon N^{1/2}(pt)^{3/4} \cdot \frac{K^{3/4}}{N^{1/4}(pt)^{1/4}}.
\end{equation}
Notice that for $N> (pt)$, which is all we have to consider, the contribution from \eqref{eq:contribution-0equal} dominates the contribution from \eqref{eq:contribution-diff}.
A similar calculation for the small $q$ regime yields,
\begin{equation}
	S_{\text{main},0\neq} \ll_\varepsilon pK^{3/2} \ll N^{1/2} (pt)^{3/4} \frac{K^{3/4}}{(pt)^{3/8}} \frac{p}{N^{1/2}}.
\end{equation}
We can compare this with \eqref{eq:contribution-0equal} and note that for $N \gg (pt)^{4/3}$ and $p \leq t^2$ it is smaller. 

\subsubsection{Non-zero frequencies: small moduli}
For large values of the modulus $C$, Lemma \ref{lem:Ibound} provides a strong bound for the integral $\mathfrak{I}$, and the corresponding terms will be treated in the next section. Here, we treat the case of small moduli and therefore assume $C \ll t^\delta \sqrt{NK}/t$  or $C \ll t^{\delta} \sqrt{N/Kt}$.

We use the bound $\mathfrak{I} \ll t^{-1}$ from Lemma \ref{lem:Ibound} and the bound on $\mathfrak{C}$ given in \eqref{CNonzeroBound} so that 
\begin{equation}
\Omega_{\text{non}} \ll \frac{N_0}{tn_1^2} \sum_{\substack{n_2 \neq 0 \\ n_2 \ll N_1}} \sum_{\substack{q, q'\sim C \\ n_1\mid q, q'}} \frac{1}{qq'} {{\sums_{\substack{a \mod q \\ a' \mod q'}}}} \sum_{\substack{m, m' \ll M_0 \\ {m \equiv ap \mod q} \\ {m' \equiv a'p \mod q'}}}  \left(\frac{q}{n_1}, \frac{q'}{n_1}, n_2\right).
\end{equation}
Note that, for real numbers $0\leq \alpha, \beta, \gamma \leq 1$,
\begin{equation} 
\sum_{\substack{{a \leqslant X} \\ b \leqslant Y \\ c \leqslant Z}} \frac{(a,b,c)}{a^\alpha b^\beta c^\gamma} 
 \ll \sum_{d \leq \min\{X,Y,Z\}} \frac{1}{d^{\alpha+\beta+\gamma-1}} \sum_{\substack{{a \leq X/d} \\ b\leq Y/d \\ c\leq Z/d}} \ \frac{1}{a^\alpha b^\beta c^\gamma}  
 \ll X^{1-\alpha}Y^{1-\beta}Z^{1-\gamma} (XYZ)^{\varepsilon}. \label{gcd-bounds}
\end{equation}

\noindent We therefore have
\begin{equation}\label{eq:OmegaNonBoundSmallModuli}
\Omega_{\text{non}} \ll_\varepsilon \frac{N_0}{n_1^4 t} N_1 M_0^2  \ll_\varepsilon     \frac{N^{1/2}K^{3/2} }{n_1^4t} \frac{C}{K} \frac{p^2K}{N}.
\end{equation}
Plugging in the $C\ll t^\delta \sqrt{NK}/t$ bound we get that $\Omega_{\text{non}} \ll_\varepsilon  \frac{K^2p^2}{t^2n_1^4}$. Plugging in  the bound $C \ll t^\delta \sqrt{N/Kt}$ we have $\Omega_{\text{non}} \ll_\varepsilon \frac{K p^2}{t^{3/2}n_1^4}$.
The corresponding $S_{\mathrm{main}}$ contribution is.
\begin{equation}
\label{eq:contribution-non-small}
S_{\mathrm{main, non}} \ll_\varepsilon N^{1/2}(pt)^{3/4} \cdot \max\left\{ \frac{p^{1/8}  K^{5/4}}{t^{11/8}}, \frac{K^{3/4}}{(pt)^{3/8}} \frac{p^{1/2}}{t^{3/4}} \right\}.
\end{equation}
When compared with \eqref{eq:contribution-0equal} the second term in the maximum can be ignored+ $p < t^{3/2}$.

\subsubsection{Non-zero frequencies: large moduli}

In this section we assume $C \gg t^\delta \sqrt{NK}/t$.
We use the bounds \eqref{eq:Ibound} on $\mathfrak{I}$ from Lemma \ref{lem:Ibound} and \eqref{CNonzeroBound} on $\mathfrak{C}$ so that 
\begin{equation}
\Omega_{\text{non}} \ll \frac{N_0}{tn_1^2} \sum_{\substack{n_2 \neq 0 \\ n_2 \ll N_1}} \sum_{\substack{q, q'\sim C \\ n_1\mid q, q'}} \frac{1}{qq'} {{\sums_{\substack{a \mod q \\ a' \mod q'}}}} \sum_{\substack{m, m' \ll M_0 \\ {m \equiv ap \mod q} \\ {m' \equiv a'p \mod q'}}}  \left(\frac{q}{n_1}, \frac{q'}{n_1}, n_2\right) \left( \frac{qq'}{N_0n_2} \right)^{1/{2}}.
\end{equation}

Using \eqref{gcd-bounds} we have
\begin{equation}\label{eq:OmegaNonBoundLargeModuli}
\Omega_{\text{non}} \ll \frac{N_0^{1/2}N_1^{1/2}CM_0^2}{tn_1^4} \ll_\varepsilon \frac{p^2 t}{K^{3/2} n_1^4}.
\end{equation}
After inputing this bound in \eqref{eq:SmainPrimeCauchySchwartz}, applying Lemma \ref{lem:n1-sum} to deal with the $n_1$ sum, and using~\eqref{SmainBound}, we obtain that the corresponding contribution in $S_{\mathrm{main}}$ is
\begin{equation}
\label{eq:contribution-non}
S_{\mathrm{main, non}} \ll_\varepsilon \frac{N^{2/3}N_0^{1/6}p^{1/2}t^{1/2}}{K^{3/4}}  \ll_\varepsilon N^{1/2}(pt)^{3/4} \cdot \frac{(pt)^{1/8}}{K^{1/2}}.
\end{equation}

\subsection{Subconvexity bound}
\label{subsec:subconvexity}

We optimize the above bounds in $K$.  Equating the contribution from the diagonal zero frequency \eqref{eq:contribution-0equal} and the one from the nonzero frequency with large moduli \eqref{eq:contribution-non}, we obtain $K=(pt)^{2/5}$ and the corresponding contributions to $S_{\mathrm{main}}$ are bounded by
\begin{equation}
\label{eq:subconvex-bound}
N^{1/2}(pt)^{3/4} (pt)^{-3/40}.
\end{equation}
For the same $K$ value, we obtain $\Scomp\ll_\varepsilon N^{3/4} \ll_\varepsilon N^{1/2} (pt)^{3/4} (pt)^{-3/8}$ from \eqref{a}. In short, the complementary range is not the bottleneck.

Concerning the contribution of small moduli \eqref{eq:contribution-non-small}, with the choice $K=(pt)^{2/5}$ it is bounded by
\begin{equation}
\label{eq:extended-range-bound}
N^{1/2}(pt)^{3/4} \ \max\left\{ \frac{p^{5/8}}{t^{7/8}}, \frac{p^{17/40}}{t^{33/40}} \right\}
\end{equation}
and, in the range $p<t^{8/7}$, it also satisfies the same bound \eqref{eq:subconvex-bound}. Altogether we have
\begin{equation}
S(N) \ll_\varepsilon N^{1/2}(pt)^{3/4} (pt)^{-3/40}, 
\end{equation}
completing the proof of Theorem \ref{theorem}.

\smallskip

\noindent\textit{Remark.} We can allow for the larger range $p<t^{15/13}$, for which the bound \eqref{eq:extended-range-bound} is still subconvex.

\bibliographystyle{acm}
\bibliography{GL3subconvexity}

\end{document}